\documentclass[12pt,twoside]{amsart}
\usepackage{amsmath}
\usepackage{amsthm}
\usepackage{amsfonts}
\usepackage{amssymb}
\usepackage{latexsym}
\usepackage[all]{xy}

\date{}
\allowdisplaybreaks[4] \footskip=15pt
\renewcommand{\uppercasenonmath}[1]{}

\numberwithin{equation}{section} \theoremstyle{plain}
\newtheorem*{thm*}{Main Theorem}
\newtheorem{thm}{Theorem}[section]
\newtheorem{cor}[thm]{Corollary}
\newtheorem*{cor*}{Corollary}
\newtheorem{lem}[thm]{Lemma}
\newtheorem*{lem*}{Lemma}

\newtheorem*{prop*}{Proposition}
\newtheorem{rem}[thm]{Remark}
\newtheorem*{rem*}{Remark}

\newtheorem*{exa*}{Example}

\newtheorem*{df*}{Definition}

\newtheorem*{conj*}{Conjecture}
\newtheorem*{ack*}{ACKNOWLEDGEMENTS}

\newcommand{\pf}{\noindent\begin {proof}}
\newcommand{\epf}{\end{proof}}

\newcommand{\Ext}{\mbox{\rm Ext}}

\newcommand{\Hom}{\mbox{\rm Hom}}
\newcommand{\Tor}{\mbox{\rm Tor}}
\newcommand{\im}{\mbox{\rm im}}

\begin{document}
\begin{center}
{\large \bf Ding modules and dimensions over formal triangular
matrix rings}

\vspace{0.5cm} Lixin Mao\\
%\bigskip
Department of Mathematics and Physics, Nanjing Institute of
Technology,\\ Nanjing 211167, China\\
E-mail: maolx2@hotmail.com \\
\end{center}
%\begin{figure}[b]
%\rule[-2.5truemm]{5cm}{0.1truemm}\\[2mm]
%{\small }
%\end{figure}

%\begin{figure}[b]
%\rule[-2.5truemm]{5cm}{0.1truemm}\\[2mm]
%{\small }
%\end{figure}
\bigskip
\centerline { \bf  Abstract}
 \bigskip
\leftskip10truemm \rightskip10truemm
 \noindent
Let $T=\biggl(\begin{matrix} A&0\\
U&B \end{matrix}\biggr)$ be a formal triangular matrix ring, where
$A$ and $B$ are rings and $U$ is a $(B, A)$-bimodule. We
prove that: (1) If $U_{A}$ and $_{B}U$ have finite flat dimensions, then a  left $T$-module $\biggl(\begin{matrix} M_{1}\\
M_{2}\end{matrix}\biggr)_{\varphi^{M}}$ is Ding projective if and
only if $M_{1}$ and $M_{2}/\im(\varphi^{M})$ are Ding projective and
the morphism $\varphi^{M}$ is a monomorphism. (2) If $T$ is a right
coherent ring, $_{B}U$ has finite flat dimension, $U_{A}$ is
finitely presented and  has finite projective or $FP$-injective
dimension, then a right $T$-module $(W_{1}, W_{2})_{\varphi_{W}}$ is
Ding injective if and only if $W_{1}$ and
$\ker(\widetilde{{\varphi_{W}}})$ are Ding injective and the
morphism $\widetilde{{\varphi_{W}}}$ is an epimorphism. As a
consequence, we describe Ding projective and Ding injective dimensions of a $T$-module.\\
\vbox to 0.3cm{}

\leftskip0truemm \rightskip0truemm
\bigskip
\section { \bf Introduction}
\bigskip
The origin of Gorenstein homological algebra may date back to 1960s
when Auslander and Bridger introduced the concept of G-dimension for
finitely generated modules over a two-sided Noetherian ring
\cite{AB}. In 1990s, Enochs and Jenda extended the ideas of
Auslander and Bridger and introduced the concepts of Gorenstein
projective and Gorenstein  injective modules over arbitrary rings
\cite{EJ1}. In \cite{DLM, MD}, Ding, Li and Mao considered two
special cases of the Gorenstein projective and Gorenstein injective
modules, which they called strongly Gorenstein flat and Gorenstein
$FP$-injective modules respectively. These two classes of modules
over coherent rings possess many nice properties analogous to
Gorenstein projective and Gorenstein injective modules over
Noetherian rings (see \cite{DLM, G1, MD, Y, YLL}). So Gillespie
later renamed strongly Gorenstein flat as Ding projective, and
Gorenstein $FP$-injective as Ding injective (see \cite{G1} for
details).

Let $A$ and $B$ be rings and $U$ be a $(B, A)$-bimodule. $T=\biggl(\begin{matrix} A&0\\
U&B \end{matrix}\biggr)$ is known as a \emph{formal triangular
matrix ring} with usual matrix addition and multiplication. Formal
triangular matrix rings play an important role in ring theory and
the representation theory of algebra \cite{ARS}. This kind of rings
are often used to construct examples and counterexamples, which make
the theory of rings and modules more abundant and concrete. So the
properties of formal triangular matrix rings and modules over them
have deserved more and more interests (see \cite{AS, ARS, EIT, FGR},
\cite{G2}-\cite{HV1}, \cite{KT, LC, Z, ZLW}). For example, Zhang
\cite{Z} explicitly described the Gorenstein projective modules over
a triangular matrix Artin algebra. Enochs and other authors
\cite{EIT} characterized when a left module over a triangular matrix
ring is Gorenstein projective or Gorenstein injective under the
``Gorenstein regular" condition. Zhu, Liu and Wang \cite{ZLW} also
investigated Gorenstein homological dimensions of modules over
triangular matrix rings under the ``Gorenstein regular" condition.

The present paper is devoted to  Ding projective and Ding injective
modules and dimensions over formal triangular matrix rings.

In Section 3, let  $U_{A}$ and $_{B}U$ have finite flat dimensions, we prove that a  left $T$-module $M=\biggl(\begin{matrix} M_{1}\\
M_{2}\end{matrix}\biggr)_{\varphi^{M}}$ is Ding projective if and
only if $M_{1}$ is a Ding projective left $A$-module,
$M_{2}/\im(\varphi^{M})$ is a Ding projective left $B$-module and
the morphism $\varphi^{M}$ is a monomorphism. As a consequence, we
prove that, if $B$ has finite left global Ding projective dimension,
$U_{A}$ has finite flat dimension and $_{B}U$ is projective, then
max$\{Dpd(M_{1}),Dpd(M_{2})\}\leq Dpd(M)\leq$
max$\{Dpd(M_{1})+1,Dpd(M_{2})\}$.

In Section 4, let $T$ be a right coherent ring, $_{B}U$ have finite
flat dimension, $U_{A}$  be finitely presented  and have finite
projective or $FP$-injective dimension, we obtain that  a right
$T$-module $W=(W_{1}, W_{2})_{\varphi_{W}}$ is Ding injective if and
only if $W_{1}$ is a Ding injective right $A$-module,
$\ker(\widetilde{{\varphi_{W}}})$ is a Ding injective right
$B$-module and the morphism $\widetilde{{\varphi_{W}}}$ is an
epimorphism. As a consequence, we get that, if $T$ is a right
coherent ring, $B$ has finite right global Ding injective dimension,
$_{B}U$ is flat, $U_{A}$ is finitely presented  and has finite
projective or $FP$-injective dimension, then
max$\{Did(W_{1}),Did(W_{2})\}\leq Did(W)\leq$
max$\{Did(W_{1})+1,Did(W_{2})\}$.
\bigskip
\section {\bf Preliminaries}

Throughout this paper, all rings are nonzero associative rings with
identity and all modules are unitary. For a ring $R$, we write
$R$-Mod (resp. Mod-$R$) for the category of left (resp. right) $R$-modules. $_RM$ (resp. $M_{R}$) denotes a left (resp. right)
$R$-module. The character module $\Hom_{\mathbb{Z}}(M, \mathbb{Q}/\mathbb{Z})$
of a module $M$  is denoted by $M^{+}$. $pd(M)$, $id(M)$ and $fd(M)$
denote the projective, injective and flat dimensions of a module $M$
respectively. $T=\biggl(\begin{matrix} A&0\\U&B \end{matrix}\biggr)$ always means a formal triangular matrix
ring, where $A$ and $B$ are rings and $U$ is a $(B, A)$-bimodule.

A left $R$-module $M$ is called \emph{Ding projective} \cite{DLM,
G1} (resp. \emph{Gorenstein projective} \cite{EJ1}) if there is an
exact sequence $\cdots\rightarrow P^{-2}\rightarrow
P^{-1}\rightarrow P^{0}\rightarrow P^{1}\rightarrow \cdots$ of
projective left $R$-modules with $M=\ker(P^{0}\rightarrow P^{1})$,
which remains exact after applying $\Hom_{R}(-, B)$ for any flat
(resp. projective)  left $R$-module $B$.

A right $R$-module $X$ is called \emph{$FP$-injective} \cite{S} if
$\Ext_{R}^1(N, X)=0$ for every finitely presented right $R$-module
$N$. The \emph{$FP$-injective dimension} of a right $R$-module $X$,
denoted by $FP$-$id(X)$, is defined to be the smallest integer
$n\geq 0$ such that $\Ext_{R}^{n+1}(N,X)=0$ for every finitely
presented right $R$-module $N$ (if no such $n$ exists, set
$FP$-$id(X)=\infty$).

A right  $R$-module $N$ is called \emph{Ding injective} \cite{G1,
MD} (resp. \emph{Gorenstein injective} \cite{EJ1}) if there is an
exact sequence $\cdots\rightarrow E^{-1}\rightarrow E^{0}\rightarrow
E^{1}\rightarrow E^{2}\rightarrow \cdots$ of injective right
$R$-modules with $N=\ker(E^{0}\rightarrow E^{1})$, which remains
exact after applying $\Hom_{R}(G, -)$ for any $FP$-injective (resp.
injective) right $R$-module $G$.

By \cite[Theorem 1.5]{G}, the category $T$-Mod of left $T$-modules
is equivalent to the category $\Omega$ whose objects are triples
$M=\biggl(\begin{matrix}M_{1}\\M_{2}\end{matrix}\biggr)_{\varphi^{M}}$,
where $M_{1}\in A$-Mod, $M_{2}\in B$-Mod and $\varphi^{M}:
U\otimes_{A} M_{1} \rightarrow M_{2}$ is a $B$-morphism, and whose
morphisms from
$\biggl(\begin{matrix} M_{1}\\M_{2}\end{matrix}\biggr)_{\varphi^{M}}$ to $\biggl(\begin{matrix} N_{1}\\
N_{2}\end{matrix}\biggr)_{\varphi^{N}}$ are pairs
$\biggl(\begin{matrix} f_{1}\\f_{2}\end{matrix}\biggr)$  such that
$f_{1}\in\Hom_{A}(M_{1}, N_{1}), f_{2}\in\Hom_{B}(M_{2}, N_{2})$
satisfying that the following diagram  is commutative. $$\xymatrix{U\otimes_{A} M_{1}\ar[d]_{\varphi^{M}}\ar[r]^{1\otimes f_{1}}&U\otimes_{A} N_{1}\ar[d]^{\varphi^{N}}\\
M_{2}\ar[r]^{f_{2}}&N_{2}.}$$ Given a triple $M=\biggl(\begin{matrix} M_{1}\\
M_{2}\end{matrix}\biggr)_{\varphi^{M}}$ in $\Omega$, we shall denote
by $\widetilde{{\varphi^{M}}}$  the $A$-morphism from $M_{1}$ to
$\Hom_{B}(U, M_{2})$ given by $\widetilde{{\varphi^{M}}}(x)(u)
=\varphi^{M} (u\otimes x)$ for each $u \in U$ and $x\in M_{1}$.

Note that a sequence $0\rightarrow\biggl(\begin{matrix} M'_{1}\\
M'_{2}\end{matrix}\biggr)_{\varphi^{M'}}\rightarrow\biggl(\begin{matrix} M_{1}\\
M_{2}\end{matrix}\biggr)_{\varphi^{M}}\rightarrow\biggl(\begin{matrix} M''_{1}\\
M''_{2}\end{matrix}\biggr)_{\varphi^{M''}}\rightarrow 0$ of left
$T$-modules is exact if and only if both sequences $0\rightarrow
M'_{1}\rightarrow M_{1}\rightarrow M''_{1}\rightarrow 0$ and
$0\rightarrow M'_{2}\rightarrow M_{2}\rightarrow M''_{2}\rightarrow
0$ are exact.

Recall that the \emph{product category} $A$-Mod $\times B$-Mod is
defined as follows: An object of $A$-Mod $\times B$-Mod is a pair
$(M,N)$ with $M\in A$-Mod and $N\in B$-Mod, a morphism from $(M,N)$
to $(M',N')$ is a pair $(f,g)$ with $f\in\Hom_{A}(M,M')$ and
$g\in\Hom_{B}(N,N')$.

There are some functors between the category $T$-Mod and the product
category $A$-Mod $\times B$-Mod as follows:

(1) $\textbf{p}: A$-Mod $\times B$-Mod $\rightarrow T$-Mod is
defined as follows: for each object $(M_{1}, M_{2})$ of $A$-Mod
$\times B$-Mod, let $\textbf{p}(M_{1}, M_{2}) =\biggl(\begin{matrix} M_{1}\\
(U\otimes_{A} M_{1})\oplus M_{2}\end{matrix}\biggr)$ with the
obvious map and for any morphism $(f_{1}, f_{2})$ in $A$-Mod $\times B$-Mod, let $\textbf{p}(f_{1}, f_{2}) =\biggl(\begin{matrix} f_{1}\\
(1\otimes_{A} f_{1})\oplus f_{2}\end{matrix}\biggr)$.

(2) $\textbf{h}: A$-Mod $\times B$-Mod $\rightarrow T$-Mod is
defined as follows: for each object $(M_{1}, M_{2})$ of $A$-Mod
$\times B$-Mod, let $\textbf{h}(M_{1}, M_{2}) =\biggl(\begin{matrix} M_{1}\oplus\Hom_{B}(U,M_{2})\\
M_{2}\end{matrix}\biggr)$ with the obvious map and for any morphism
$(f_{1}, f_{2})$ in $A$-Mod $\times B$-Mod, let $h(f_{1}, f_{2})=\biggl(\begin{matrix} f_{1}\oplus\Hom_{B}(U,f_{2})\\
f_{2}\end{matrix}\biggr)$.

(3) $\textbf{q}: T$-Mod $\rightarrow A$-Mod $\times B$-Mod is
defined, for each left $T$-module $\biggl(\begin{matrix} M_{1}\\
M_{2}\end{matrix}\biggr)$ as $\textbf{q} \biggl(\begin{matrix} M_{1}\\
M_{2}\end{matrix}\biggr)=(M_{1}, M_{2})$, and for each morphism
$\biggl(\begin{matrix} f_{1}\\f_{2}\end{matrix}\biggr)$ in $T$-Mod as $\textbf{q} \biggl(\begin{matrix} f_{1}\\
f_{2}\end{matrix}\biggr)=(f_{1}, f_{2})$.

It is easy to see that \textbf{p} is a left adjoint of \textbf{q, h}
is a right adjoint of \textbf{q}.

Analogously, the category Mod-$T$ of right $T$-modules is equivalent
to the category $\Gamma$ whose objects are triples
$W=(W_{1},W_{2})_{\varphi_{W}}$, where $W_{1}\in$ Mod-$A$,
$W_{2}\in$ Mod-$B$ and $\varphi_{W}: W_{2}\otimes_{B} U \rightarrow
W_{1}$ is an $A$-morphism, and whose morphisms from
$(W_{1},W_{2})_{\varphi_{W}}$ to $(X_{1},X_{2})_{\varphi_{X}}$ are
pairs $(g_{1},g_{2})$  such that $g_{1}\in\Hom_{A}(W_{1}, X_{1}),
g_{2}\in\Hom_{B}(W_{2}, X_{2})$  and $\varphi_{X}(g_{2}\otimes
1)=g_{1}\varphi_{W}$. Given such a triple
$W=(W_{1},W_{2})_{\varphi_{W}}$ in $\Gamma$, we shall denote by
$\widetilde{{\varphi_{W}}}$ the $B$-morphism from $W_{2}$ to
$\Hom_{A}(U, W_{1})$ given by $\widetilde{{\varphi_{W}}}(y)(u)
=\varphi_{W} (y\otimes u)$ for each $u \in U$ and $y\in W_{2}$.
There exist similar functors $\textbf{p},\textbf{q},\textbf{h}$
between the category Mod-$T$ and the product category Mod-$A$
$\times$ Mod-$B$.

In the rest of the paper we shall identify $T$-Mod  (resp. Mod-$T$)
with this category $\Omega$ (resp. $\Gamma$) and, whenever there is
no possible confusion, we shall omit the morphism $\varphi^{M}$
(resp. $\varphi_{W}$).
\bigskip
\section {\bf Ding projective modules and dimensions}
\bigskip
We start with several lemmas.
\begin{lem} \label{lem: 3.1}The following conditions are equivalent:
\begin{enumerate}\item $X$ is a Ding projective left $R$-module. \item There is an exact
sequence $\cdots\rightarrow P^{-1}\rightarrow P^{0}\rightarrow
P^{1}\rightarrow P^{2}\rightarrow \cdots$ of projective left
$R$-modules with $X=\ker(P^{0}\rightarrow P^{1})$, , which remains
exact after applying $\Hom_{R}(-, G)$ for each left $R$-module $G$
with finite flat dimension.
\end{enumerate}
\end{lem}
\begin{proof}It is enough to show that $(1) \Rightarrow (2)$.
There is an exact sequence $\Lambda: \cdots\rightarrow
P^{-1}\rightarrow P^{0}\rightarrow P^{1}\rightarrow P^{2}\rightarrow
\cdots$ of projective left $R$-modules with $X=\ker(P^{0}\rightarrow
P^{1})$, which remains exact after applying  $\Hom_{R}(-, F)$ for
each flat  left $R$-module $F$. Let $fd(G)=n<\infty$. Then there is
an exact sequence $0\rightarrow F_{n}\rightarrow F_{n-1}
\rightarrow\cdots\rightarrow F_{1}\rightarrow F_{0}\rightarrow
G\rightarrow 0$ with each $F_{i}$ flat. So we get the exact sequence
of complexes
$0\rightarrow\Hom_{R}(\Lambda,F_{n})\rightarrow\cdots\rightarrow\Hom_{R}(\Lambda,F_{0})\rightarrow\Hom_{R}(\Lambda,G)\rightarrow
0$. By \cite[Theorem 6.3]{R}, $\Hom_{R}(\Lambda,G)$ is exact since
$\Hom_{R}(\Lambda,F_{n}),\cdots,\Hom_{R}(\Lambda,F_{0})$ are exact.
\end{proof}
\begin{lem} \label{lem: 3.2}Let $_{B}U$  have finite flat
dimension.\begin{enumerate}\item If $E$ is an injective right
$A$-module, then the right $B$-module $\Hom_{A}(U,E)$ has finite
injective dimension.\item If $F$ is a flat left $A$-module, then the
left $B$-module $U\otimes_{A}F$ has finite flat dimension.
\end{enumerate}
\end{lem}
\begin{proof}Let $fd(_{B}U)=n<\infty$.

(1) By \cite[p.360]{R}, for any right $B$-module $X$, we have
$$\Ext_{B}^{n+1}(X,\Hom_{A}(U,E))\cong\Hom_{A}(\Tor^{B}_{n+1}(X,U),E)=0.$$
Hence $id(\Hom_{A}(U,E))\leq n<\infty$.

(2) By \cite[Theorem 9.48]{R}, for any right $B$-module $X$, we have
$$\Tor_{n+1}^{B}(X, U\otimes_{A} F)\cong \Tor_{n+1}^{B}(X,
U)\otimes_{A} F=0.$$ So $fd(U\otimes_{A}F)\leq n<\infty$.
\end{proof}
\begin{lem} \label{lem: 3.3}Let $M=\biggl(\begin{matrix} M_{1}\\
M_{2}\end{matrix}\biggr)_{\varphi^{M}}$ be a  left $T$-module.
\begin{enumerate}\item \cite[Theorem 3.1]{HV1} $M$ is a projective
left $T$-module if and only if $M_{1}$ is a projective left
$A$-module, $M_{2}/\im(\varphi^{M})$ is a projective left $B$-module
and $\varphi^{M}$ is a monomorphism.\item \cite[Proposition
1.14]{FGR} $M$ is a flat left $T$-module if and only if $M_{1}$ is a
flat left $A$-module, $M_{2}/\im(\varphi^{M})$ is a flat left
$B$-module and $\varphi^{M}$ is a monomorphism.
\end{enumerate}
\end{lem}
Now we describe explicitly the structure of a Ding projective left
$T$-module.
\begin{thm} \label{thm: 3.4}Let $U_{A}$ and $_{B}U$ have finite flat dimensions and $M=\biggl(\begin{matrix} M_{1}\\
M_{2}\end{matrix}\biggr)_{\varphi^{M}}$ be a  left $T$-module. The
following  conditions are equivalent:\begin{enumerate}\item $M$ is a
Ding projective left $T$-module.\item $M_{1}$ is a Ding projective
left $A$-module, $M_{2}/\im(\varphi^{M})$ is a Ding projective left
$B$-module and $\varphi^{M}$ is a monomorphism.
\end{enumerate}
In this case, $U\otimes_{A} M_{1}$ is Ding projective left
$B$-module if and only if $M_{2}$ is Ding projective left
$B$-module.
\end{thm}
\begin{proof}(1) $\Rightarrow$ (2) There is an exact sequence  of
projective left $T$-modules $$\Delta: \cdots\rightarrow\biggl(\begin{matrix} P^{-1}_{1}\\
P^{-1}_{2}\end{matrix}\biggr)_{\varphi^{-1}}\stackrel{\Small\biggl(\begin{matrix} \partial^{-1}_{1}\\
\partial^{-1}_{2}\end{matrix}\biggr)}\rightarrow\biggl(\begin{matrix} P^{0}_{1}\\
P^{0}_{2}\end{matrix}\biggr)_{\varphi^{0}}\stackrel{\Small\biggl(\begin{matrix} \partial^{0}_{1}\\
\partial^{0}_{2}\end{matrix}\biggr)}\rightarrow\biggl(\begin{matrix} P^{1}_{1}\\
P^{1}_{2}\end{matrix}\biggr)_{\varphi^{1}}\stackrel{\Small\biggl(\begin{matrix} \partial^{1}_{1}\\
\partial^{1}_{2}\end{matrix}\biggr)}\rightarrow\biggl(\begin{matrix} P^{2}_{1}\\
P^{2}_{2}\end{matrix}\biggr)_{\varphi^{2}}\rightarrow\cdots$$
with $M=\ker\biggl(\begin{matrix} \partial^{0}_{1}\\
\partial^{0}_{2}\end{matrix}\biggr)$, which remains
exact after applying $\Hom_{T}(-, H)$  for each flat left $T$-module
$H$. By Lemma \ref{lem: 3.3}, we get the   exact sequence
$$\Lambda_{1}: \cdots\rightarrow
P_{1}^{-1}\stackrel{\Small\partial_{1}^{-1}}\rightarrow
P_{1}^{0}\stackrel{\Small\partial_{1}^{0}}\rightarrow
P_{1}^{1}\stackrel{\Small\partial_{1}^{1}}\rightarrow
P_{1}^{2}\rightarrow \cdots$$  of projective left $A$-modules with
$M_{1}=\ker(\partial_{1}^{0})$.

Let $F$ be a flat left $A$-module. There exists the exact sequence in $T$-Mod $$0\rightarrow\biggl(\begin{matrix} 0\\
U\otimes_{A} F\end{matrix}\biggr)\rightarrow\biggl(\begin{matrix} F\\
U\otimes_{A} F\end{matrix}\biggr)\rightarrow\biggl(\begin{matrix}F\\
0\end{matrix}\biggr)\rightarrow 0,$$ which induces the exact sequence of complexes $$0\rightarrow\Hom_{T}(\Delta,\biggl(\begin{matrix} 0\\
U\otimes_{A} F\end{matrix}\biggr))\rightarrow\Hom_{T}(\Delta,\biggl(\begin{matrix} F\\
U\otimes_{A} F\end{matrix}\biggr))\rightarrow\Hom_{T}(\Delta,\biggl(\begin{matrix}F\\
0\end{matrix}\biggr))\rightarrow 0.$$ By Lemma \ref{lem: 3.3}, $\biggl(\begin{matrix} F\\
U\otimes_{A} F\end{matrix}\biggr)$ is flat, hence $\Hom_{T}(\Delta,\biggl(\begin{matrix} F\\
U\otimes_{A} F\end{matrix}\biggr))$  is exact. Since
$fd(_{B}U)<\infty$, $fd(U\otimes_{A}F)<\infty$ by Lemma \ref{lem: 3.2} and so $fd\biggl(\begin{matrix} 0\\
U\otimes_{A} F\end{matrix}\biggr)<\infty$. By Lemma \ref{lem: 3.1}, $\Hom_{T}(\Delta,\biggl(\begin{matrix} 0\\
U\otimes_{A} F\end{matrix}\biggr))$ is exact and so $\Hom_{T}(\Delta, \biggl(\begin{matrix}F\\
0\end{matrix}\biggr))$ is exact by \cite[Theorem 6.3]{R}. By
adjointness of functors $\textbf{p}$ and $\textbf{q}$, we
have $\Hom_{A}(\Lambda_{1}, F)\cong\Hom_{T}(\Delta, \biggl(\begin{matrix}F\\
0\end{matrix}\biggr))$ is exact. Thus $M_{1}$ is a Ding projective
left $A$-module.

Let $\lambda_{1}: M_{1}\rightarrow P^{0}_{1}$ and $\lambda_{2}: M_{2}\rightarrow P^{0}_{2}$  be the inclusions.
Consider the following commutative diagram in $B$-Mod: $$\xymatrix{U\otimes_{A} M_{1}\ar[d]_{\varphi^{M}}\ar[r]^{1\otimes \lambda_{1}}&U\otimes_{A} P^{0}_{1}\ar[d]^{\varphi^{0}}\\
M_{2}\ar[r]^{\lambda_{2}}&P^{0}_{2}.}$$ Since $fd(U_{A})<\infty$,
$U\otimes_{A}\Lambda_{1}$ is exact by \cite[Lemma 2.3]{EIT}. Thus
$1\otimes \lambda_{1}$ is a monomorphism. Also $\varphi^{0}$ is a
monomorphism by Lemma \ref{lem: 3.3}, so $\varphi^{M}$ is a
monomorphism.

For any $i\in\mathbb{Z}$, there exists $\overline{\partial^{i}_{2}}:
P_{2}^{i}/\im(\varphi^{i})\rightarrow
P_{2}^{i+1}/\im(\varphi^{i+1})$ such that  the following diagram
with exact rows is commutative.
$$\xymatrix{&\vdots\ar[d]&&\vdots\ar[d] &\vdots\ar@{.>}[d] &\\
0 \ar[r]
&U\otimes_{A}P_{1}^{-1}\ar[d]^{^{1\otimes\partial^{-1}_{1}}}\ar[rr]^{\varphi^{-1}}&&P_{2}^{-1}\ar[d]^{\partial^{-1}_{2}}
\ar[r]&P_{2}^{-1}/\im(\varphi^{-1})\ar[r]\ar@{.>}[d]^{\overline{\partial^{-1}_{2}}}&0\\0
\ar[r] &U\otimes_{A}
P_{1}^{0}\ar[d]^{^{1\otimes\partial^{0}_{1}}}\ar[rr]^{\varphi^{0}}&&P_{2}^{0}\ar[d]^{\partial^{0}_{2}}
\ar[r]&P_{2}^{0}/\im(\varphi^{0})\ar[r]\ar@{.>}[d]^{\overline{\partial^{0}_{2}}}&0\\0
\ar[r] &U\otimes_{A}
P_{1}^{1}\ar[d]^{^{1\otimes\partial^{1}_{1}}}\ar[rr]^{\varphi^{1}}&&P_{2}^{1}\ar[d]^{\partial^{1}_{2}}
\ar[r]&P_{2}^{1}/\im(\varphi^{1})\ar[r]\ar@{.>}[d]^{\overline{\partial^{1}_{2}}}&0\\
0\ar[r]&U\otimes_{A} P_{1}^{2}\ar[d]\ar[rr]^{\varphi^{2}}&&P_{2}^{2}
\ar[d]\ar[r]&P_{2}^{2}/\im(\varphi^{2})\ar@{.>}[d]\ar[r]&0\\
&\vdots &&\vdots &\vdots &}$$ Since the first column and the second
column are exact, we get the exact sequence
$$\Xi: \cdots\rightarrow
P_{2}^{-1}/\im(\varphi^{-1})\stackrel{\Small
\overline{\partial^{-1}_{2}}}\rightarrow
P_{2}^{0}/\im(\varphi^{0})\stackrel{\Small
\overline{\partial^{0}_{2}}}\rightarrow
P_{2}^{1}/\im(\varphi^{1})\stackrel{\Small
\overline{\partial^{1}_{2}}}\rightarrow
P_{2}^{2}/\im(\varphi^{2})\rightarrow \cdots$$ of projective left
$B$-modules  by \cite[Theorem 6.3]{R} with
$M_{2}/\im(\varphi^{M})\cong\ker(\overline{{\partial^{0}_{2}}})$.

Let $G$ be a flat left $B$-module. Since $\biggl(\begin{matrix} 0\\
G\end{matrix}\biggr)$ is a flat left $T$-module,
$\Hom_{T}(\Delta, \biggl(\begin{matrix} 0\\
G\end{matrix}\biggr))$ is exact. Thus by adjointness of
functors $\textbf{p}$ and $\textbf{q}$, $\Hom_{B}(\Xi, G)\cong\Hom_{T}(\Delta, \biggl(\begin{matrix} 0\\
G\end{matrix}\biggr))$ is exact. So $M_{2}/\im(\varphi^{M})$ is a
Ding projective left $B$-module.

(2) $\Rightarrow$ (1) Since $\varphi^{M}: U\otimes_{A} M_{1}\rightarrow M_{2}$ is a monomorphism, there exists an exact sequence
in $T$-Mod $$0\rightarrow\biggl(\begin{matrix} M_{1}\\
U\otimes_{A} M_{1}\end{matrix}\biggr)\rightarrow\biggl(\begin{matrix} M_{1}\\
M_{2}\end{matrix}\biggr)_{\varphi^{M}}\rightarrow\biggl(\begin{matrix} 0\\
M_{2}/\im(\varphi^{M})\end{matrix}\biggr)\rightarrow 0.$$

We first prove that $\biggl(\begin{matrix} M_{1}\\
U\otimes_{A} M_{1}\end{matrix}\biggr)$ is a Ding projective left
$T$-module. Since $M_{1}$ is a Ding projective left $A$-module,
there is an exact sequence
$$\Lambda: \cdots\rightarrow
P^{-1}\stackrel{\Small\partial^{-1}}\rightarrow
P^{0}\stackrel{\Small\partial^{0}}\rightarrow
P^{1}\stackrel{\Small\partial^{1}}\rightarrow P^{2}\rightarrow
\cdots$$ of projective left $A$-modules with
$M_{1}=\ker(\partial^{0})$, which remains exact after applying
$\Hom_{A}(-, F)$ for each flat left $A$-module $F$. Since
$fd(U_{A})<\infty$,
$U\otimes_{A}\Lambda$ is exact by \cite[Lemma 2.3]{EIT}. So we get the exact sequence of projective left $T$-modules
$$\Upsilon: \cdots\rightarrow\biggl(\begin{matrix} P^{-1}_{1}\\
U\otimes_{A} P^{-1}\end{matrix}\biggr)\stackrel{\Small\biggl(\begin{matrix} \partial^{-1}\\
1\otimes\partial^{-1}\end{matrix}\biggr)}\rightarrow\biggl(\begin{matrix} P^{0}\\
U\otimes_{A} P^{0}\end{matrix}\biggr)\stackrel{\Small\biggl(\begin{matrix} \partial^{0}\\
1\otimes\partial^{0}\end{matrix}\biggr)}\rightarrow\biggl(\begin{matrix} P^{1}\\
U\otimes_{A} P^{1}\end{matrix}\biggr)\rightarrow\cdots$$ with $\biggl(\begin{matrix} M_{1}\\
U\otimes_{A} M_{1}\end{matrix}\biggr)\cong\ker\biggl(\begin{matrix} \partial^{0}\\
1\otimes\partial^{0}\end{matrix}\biggr)$. For any flat left $T$-module $\biggl(\begin{matrix}H_{1}\\
H_{2}\end{matrix}\biggr)_{\varphi^{H}}$, $H_{1}$ is a flat left
$A$-module by Lemma \ref{lem: 3.3}. Then $\Hom_{T}(\Upsilon, \biggl(\begin{matrix}H_{1}\\
H_{2}\end{matrix}\biggr))\cong\Hom_{A}(\Lambda,H_{1})$ is exact by
adjointness of functors $\textbf{p}$ and $\textbf{q}$. So $\biggl(\begin{matrix} M_{1}\\
U\otimes_{A} M_{1}\end{matrix}\biggr)$ is a Ding projective left
$T$-module.

Next we prove that $\biggl(\begin{matrix} 0\\
M_{2}/\im(\varphi^{M})\end{matrix}\biggr)$ is  a Ding projective
left $T$-module. Since $M_{2}/\im(\varphi^{M})$ is a Ding projective
left $B$-module, there is an exact sequence $$\Theta:
\cdots\rightarrow Q^{-1}\stackrel{f^{-1}}\rightarrow
Q^{0}\stackrel{f^{0}}\rightarrow Q^{1}\stackrel{f^{1}}\rightarrow
Q^{2}\stackrel{f^{2}}\rightarrow \cdots$$ of projective left
$B$-modules with $M_{2}/\im(\varphi^{M})=\ker(f^{0})$, which remains
exact after applying $\Hom_{B}(-, G)$ for each flat left $B$-module
$G$. So we get the exact sequence of projective left $T$-modules
$$\biggl(\begin{matrix} 0\\
\Theta\end{matrix}\biggr): \cdots\rightarrow\biggl(\begin{matrix} 0\\
Q^{-1}\end{matrix}\biggr)\stackrel{\Small\biggl(\begin{matrix} 0\\
f^{-1}\end{matrix}\biggr)}\rightarrow\biggl(\begin{matrix}0\\
Q^{0}\end{matrix}\biggr)\stackrel{\Small\biggl(\begin{matrix} 0\\
f^{0}\end{matrix}\biggr)}\rightarrow\biggl(\begin{matrix} 0\\
Q^{1}\end{matrix}\biggr)\stackrel{\Small\biggl(\begin{matrix} 0\\
f^{1}\end{matrix}\biggr)}\rightarrow\biggl(\begin{matrix} 0\\
Q^{2}\end{matrix}\biggr)\rightarrow\cdots$$ with $\biggl(\begin{matrix} 0\\
M_{2}/\im(\varphi^{M})\end{matrix}\biggr)=\ker\biggl(\begin{matrix} 0\\
f^{0}\end{matrix}\biggr).$  Let $H=\biggl(\begin{matrix}H_{1}\\
H_{2}\end{matrix}\biggr)_{\varphi^{H}}$ be a flat left $T$-module.
By Lemma \ref{lem: 3.3}, there is the exact sequence of left
$B$-modules $$0\rightarrow U\otimes_{A} H_{1}\stackrel{\varphi^{H}}
\rightarrow H_{2}\rightarrow H_{2}/\im(\varphi^{H})\rightarrow 0$$
with $H_{1}$ and
$H_{2}/\im(\varphi^{H})$ flat. By Lemma \ref{lem: 3.2}$, fd(U\otimes_{A} H_{1})<\infty$ and so $fd(H_{2})<\infty$.  Hence $\Hom_{T}(\biggl(\begin{matrix} 0\\
\Theta\end{matrix}\biggr),\biggl(\begin{matrix}  H_{1}\\
H_{2}\end{matrix}\biggr))\cong\Hom_{B}(\Theta,
H_{2})$ is exact by Lemma \ref{lem: 3.1}, and so $\biggl(\begin{matrix} 0\\
M_{2}/\im(\varphi^{M})\end{matrix}\biggr)$  is a Ding projective
left $T$-module.

By \cite[Lemma 2.4]{YLL}, $M=\biggl(\begin{matrix} M_{1}\\
M_{2}\end{matrix}\biggr)_{\varphi^{M}}$ is a Ding projective left
$T$-module.

Finally, if the equivalent conditions above hold, then there exists
an exact sequence $$0\rightarrow U\otimes_{A}
M_{1}\stackrel{\varphi^{M}}\rightarrow M_{2}\rightarrow
M_{2}/\im(\varphi^{M})\rightarrow 0.$$ Since
$M_{2}/\im(\varphi^{M})$ is a Ding projective left $B$-module,
$U\otimes_{A} M_{1}$ is  Ding projective if and only if $M_{2}$ is
Ding projective by \cite[Theorem 2.6]{YLL}.
\end{proof}
As an immediate consequence of Theorem \ref{thm: 3.4}, we have
\begin{cor} \label{cor: 3.5}Let $R$ be a ring, $T(R)=\biggl(\begin{matrix} R&0\\
R&R \end{matrix}\biggr)$  and $M=\biggl(\begin{matrix} M_{1}\\
M_{2}\end{matrix}\biggr)_{\varphi^{M}}$ be a  left $T(R)$-module.
The following  conditions are equivalent:\begin{enumerate}\item $M$
is a Ding projective left $T(R)$-module.\item $M_{1}$ and
$M_{2}/\im(\varphi^{M})$ are Ding projective  left $R$-modules, and
$\varphi^{M}$ is a monomorphism.\item $M_{2}$ and
$M_{2}/\im(\varphi^{M})$ are Ding projective  left $R$-modules, and
$\varphi^{M}$ is a monomorphism.
\end{enumerate}
\end{cor}
Next we investigate Ding projective dimensions of modules over
formal triangular matrix rings.

Given a left $R$-module $X$, let $Dpd(X)$ denote $\inf\{n$: there
exists an exact sequence $0\rightarrow G_{n}\rightarrow\cdots
\rightarrow G_{1}\rightarrow G_{0}\rightarrow X\rightarrow 0$ of
left $R$-modules with each $G_{i}$ Ding projective\} and call
$Dpd(X)$ the  \emph{Ding projective dimension} of $X$ \cite{DLM}. If
no such $n$ exists, set $Dpd(X)$ = $\infty$. Put $lDPD(R)$ =
sup\{$Dpd(X): X$ is any left $R$-module\} and call $lDPD(R)$  the
\emph{left global Ding projective dimension} of $R$.
\begin{lem} \label{lem: 3.6}The following  conditions  are equivalent for a left $R$-module $M$:
\begin{enumerate}\item Dpd$(M)\leq n$. \item For any  exact sequence
$0\rightarrow K_{n}\rightarrow P_{n-1}\rightarrow\cdots\rightarrow
P_{1}\rightarrow P_{0}\rightarrow M\rightarrow 0$ with  each $P_{i}$
Ding projective, $K_{n}$ is Ding projective.
\end{enumerate}
\end{lem}
\begin{proof}(1) $\Rightarrow$ (2) There exists an exact
sequence $0\rightarrow G_{n}\rightarrow\cdots \rightarrow
G_{1}\rightarrow G_{0}\rightarrow M\rightarrow 0$ of left
$R$-modules with each $G_{i}$   Ding projective. Since the class of
Ding projective left $R$-modules is closed under extensions, kernels
of epimorphisms, direct sums, direct summands by \cite[Theorem 2.6
and Corollary 2.7]{YLL} and \cite[Remark 2.2(2)]{DLM}, $K_{n}$ is
Ding projective by \cite[Lemma 3.12]{AB}.

(2) $\Rightarrow$ (1) is clear.
\end{proof}
\begin{lem} \label{lem: 3.7}Let  $lDPD(B)<\infty$, $U_{A}$ have finite flat dimension and $_{B}U$ be projective.
If $X$ is a Ding projective left $A$-module, then $U\otimes_{A} X$
is a Ding projective left $B$-module.
\end{lem}
\begin{proof}There is an exact sequence of projective left $A$-modules $$\Lambda: \cdots\rightarrow
P^{-1}\rightarrow P^{0}\rightarrow P^{1}\rightarrow P^{2}\rightarrow
\cdots$$ with $X=\ker(P^{0}\rightarrow P^{1})$. Since $_{B}U$ is
projective, each $U\otimes_{A} P^{i}$ is projective. Since
$fd(U_{A})<\infty$, we get the exact sequence of projective left
$B$-modules $$U\otimes_{A} \Lambda: \cdots\rightarrow U\otimes_{A}
P^{-1}\rightarrow U\otimes_{A} P^{0}\rightarrow U\otimes_{A}
P^{1}\rightarrow U\otimes_{A} P^{2}\rightarrow \cdots$$ with
$U\otimes_{A} X\cong\ker(U\otimes_{A} P^{0}\rightarrow U\otimes_{A}
P^{1})$ by \cite[Lemma 2.3]{EIT}. For each flat  left $B$-module
$G$, $id(G)<\infty$ by \cite[Proposition 3.2]{DLM}. Then
$\Hom_{B}(U\otimes_{A} \Lambda, G)$ is exact  by \cite[Lemma
2.4]{EIT}. So $U\otimes_{A} X$ is a Ding projective left $B$-module.
\end{proof}
\begin{thm} \label{thm: 3.8}Let  $lDPD(B)<\infty$,  $U_{A}$ have finite flat dimension and $_{B}U$ be projective. If $M=\biggl(\begin{matrix} M_{1}\\
M_{2}\end{matrix}\biggr)_{\varphi^{M}}$ is a  left $T$-module, then
$$max\{Dpd(M_{1}),Dpd(M_{2})\}\leq Dpd(M)\leq
max\{Dpd(M_{1})+1,Dpd(M_{2})\}.$$
\end{thm}
\begin{proof}We first prove that max$\{Dpd(M_{1}),Dpd(M_{2})\}\leq Dpd(M)$.

We may assume that $Dpd(M)=m<\infty$.
There is an exact sequence in $T$-Mod $$0\rightarrow\biggl(\begin{matrix} N^{m}_{1}\\
N^{m}_{2}\end{matrix}\biggr)_{\varphi^{m}}\stackrel{\Small\biggl(\begin{matrix} \partial^{m}_{1}\\
\partial^{m}_{2}\end{matrix}\biggr)}\rightarrow\biggl(\begin{matrix} N^{m-1}_{1}\\
N^{m-1}_{2}\end{matrix}\biggr)_{\varphi^{m-1}}\rightarrow\cdots\rightarrow\biggl(\begin{matrix} N^{0}_{1}\\
N^{0}_{2}\end{matrix}\biggr)_{\varphi^{0}}\stackrel{\Small\biggl(\begin{matrix} \partial^{0}_{1}\\
\partial^{0}_{2}\end{matrix}\biggr)}\rightarrow\biggl(\begin{matrix} M_{1}\\
M_{2}\end{matrix}\biggr)_{\varphi^{M}}\rightarrow 0$$ with each $\biggl(\begin{matrix} N^{i}_{1}\\
N^{i}_{2}\end{matrix}\biggr)_{\varphi^{i}}$ Ding projective. By
Theorem \ref{thm: 3.4}, all $N_{1}^{i}$ and
$N_{2}^{i}/\im(\varphi^{i})$ are Ding projective. So each
$U\otimes_{A} N_{1}^{i}$ is Ding projective by Lemma \ref{lem: 3.7}.
Hence each $N_{2}^{i}$ is Ding projective by Theorem \ref{thm: 3.4}.
Since there exist the exact sequences $0\rightarrow
N_{1}^{m}\stackrel{\Small\partial_{1}^{m}}\rightarrow
N_{1}^{m-1}\rightarrow\cdots\rightarrow
N_{1}^{0}\stackrel{\Small\partial_{1}^{0}}\rightarrow
M_{1}\rightarrow 0$ and $0\rightarrow
N_{2}^{m}\stackrel{\Small\partial_{2}^{m}}\rightarrow
N_{2}^{m-1}\rightarrow\cdots\rightarrow
N_{2}^{0}\stackrel{\Small\partial_{2}^{0}}\rightarrow
M_{2}\rightarrow 0$, we have  $Dpd(M_{1})\leq m$ and $Dpd(M_{2})\leq
m$.

Next we prove that $Dpd(M)\leq$ max$\{Dpd(M_{1})+1,Dpd(M_{2})\}$.

We may assume that max$\{Dpd(M_{1})+1,Dpd(M_{2})\}=n<\infty$. There
exist an exact sequence $0\rightarrow
C_{n-1}\stackrel{f_{n-1}}\rightarrow
C_{n-2}\stackrel{f_{n-2}}\rightarrow\cdots\rightarrow C_{1}
\stackrel{f_{1}}\rightarrow C_{0}\stackrel{f_{0}}\rightarrow
M_{1}\rightarrow 0$ with  each $C_{i}$ a Ding projective left
$A$-module and an exact sequence $P_{0}\stackrel{g_{0}}\rightarrow
M_{2}\rightarrow 0$ with $P_{0}$ a projective left $B$-module. Write
$K_{1}^{i}=\ker(f_{i-1})$ and $\pi_{i}: C_{i}\rightarrow K_{1}^{i}$
to be the obvious epimorphisms, $i=1,2,\cdots,n-1$. Define $h_{0}:
(U\otimes_{A}C_{0})\oplus P_{0}\rightarrow M_{2}$ by $h_{0}(u\otimes
c_{0}, x_{0})=\varphi^{M}(u\otimes f_{0}(c_{0}))+g_{0}(x_{0})$ for
$u\in U, c_{0}\in C_{0}, x_{0}\in P_{0}$. Then $h_{0}$ is clearly an
epimorphism. Thus we get an exact sequence $$0\rightarrow\biggl(\begin{matrix} K^{1}_{1}\\
K^{1}_{2}\end{matrix}\biggr)_{\psi^{1}}\rightarrow\biggl(\begin{matrix} C_{0}\\
 (U\otimes_{A}C_{0})\oplus P_{0}\end{matrix}\biggr)\stackrel{\biggl(\begin{matrix} f_{0}\\
h_{0}\end{matrix}\biggr)}\rightarrow\biggl(\begin{matrix} M_{1}\\
M_{2}\end{matrix}\biggr)_{\varphi^{M}}\rightarrow 0.$$ Also, there
is an exact sequence $P_{1}\stackrel{g_{1}}\rightarrow
K^{1}_{2}\rightarrow 0$ with $P_{1}$  a projective left $B$-module.
Define $h_{1}: (U\otimes_{A}C_{1})\oplus P_{1}\rightarrow K^{1}_{2}$
by $h_{1}(u\otimes c_{1}, x_{1})=\psi^{1}(u\otimes
\pi_{1}(c_{1}))+g_{1}(x_{1})$ for $u\in U, c_{1}\in C_{1}, x_{1}\in
P_{1}$. Since $h_{1}$ is an epimorphism, we get an exact sequence $$0\rightarrow\biggl(\begin{matrix} K^{2}_{1}\\
K^{2}_{2}\end{matrix}\biggr)_{\psi^{2}}\rightarrow\biggl(\begin{matrix} C_{1}\\
(U\otimes_{A}C_{1})\oplus P_{1}\end{matrix}\biggr)\stackrel{\biggl(\begin{matrix} \pi_{1}\\
h_{1}\end{matrix}\biggr)}\rightarrow\biggl(\begin{matrix} K^{1}_{1}\\
K^{1}_{2}\end{matrix}\biggr)_{\psi^{1}}\rightarrow 0.$$
Continuing the process, we get the  exact sequence of left $T$-modules $$0\rightarrow \biggl(\begin{matrix} 0\\
K_{2}^{n-1}\end{matrix}\biggr)\rightarrow\biggl(\begin{matrix} C_{n-1}\\
 (U\otimes_{A}C_{n-1})\oplus P_{n-1}\end{matrix}\biggr)\rightarrow\cdots$$ $$\rightarrow\biggl(\begin{matrix} C_{1}\\
 (U\otimes_{A}C_{1})\oplus P_{1}\end{matrix}\biggr)\rightarrow\biggl(\begin{matrix} C_{0}\\
 (U\otimes_{A}C_{0})\oplus P_{0}\end{matrix}\biggr)\rightarrow\biggl(\begin{matrix} M_{1}\\
M_{2}\end{matrix}\biggr)_{\varphi^{M}}\rightarrow 0.$$

By Lemma \ref{lem: 3.7}, each $U\otimes_{A}C_{i}$ is  Ding
projective and so is $(U\otimes_{A}C_{i})\oplus P_{i}$. Since
$Dpd(M_{2})\leq n$,
$K_{2}^{n-1}$ is Ding projective by Lemma \ref{lem: 3.6}. Thus $\biggl(\begin{matrix} 0\\
K_{2}^{n-1}\end{matrix}\biggr)$ and $\biggl(\begin{matrix} C_{i}\\
 (U\otimes_{A}C_{i})\oplus P_{i}\end{matrix}\biggr)$ are Ding projective by Theorem \ref{thm: 3.4}. So $Dpd(M)\leq n$.
\end{proof}
The following theorem gives an estimate of the left global Ding
projective dimension of a formal triangular matrix ring.
\begin{thm} \label{thm: 3.9}Let $_{B}U\neq 0$ be projective, $U_{A}$ have finite flat dimension. Then
$$max\{lDPD(A),lDPD(B), 1\}\leq lDPD(T)\leq
max\{lDPD(A)+1,lDPD(B)\}.$$
\end{thm}
\begin{proof}We first prove that max$\{lDPD(A),lDPD(B), 1\}\leq lDPD(T)$.

We may assume that $lDPD(T)=m<\infty$.
Since $U\neq 0$, $X=\biggl(\begin{matrix} A\\
0\end{matrix}\biggr)$ is not a Ding projective left $T$-module by
Theorem \ref{thm: 3.4}. So $m\geq Dpd(X)\geq 1$.

Let $N$ be any left $B$-module. There exists an exact sequence
$$0\rightarrow K_{m}\rightarrow
P_{m-1}\rightarrow\cdots\rightarrow P_{1} \rightarrow
P_{0}\rightarrow N\rightarrow 0$$ with each $P_{i}$ a projective
left $B$-module. Then we get the exact sequence
$$ 0\rightarrow\biggl(\begin{matrix} 0\\
K_{m}\end{matrix}\biggr)\rightarrow\biggl(\begin{matrix}0\\
P_{m-1}\end{matrix}\biggr)\rightarrow\cdots\rightarrow\biggl(\begin{matrix} 0\\
P_{1}\end{matrix}\biggr)\rightarrow\biggl(\begin{matrix} 0\\
P_{0}\end{matrix}\biggr)\rightarrow\biggl(\begin{matrix} 0\\
N\end{matrix}\biggr)\rightarrow 0.$$ Since $Dpd\biggl(\begin{matrix} 0\\
N\end{matrix}\biggr)\leq lDPD(T)= m$, $\biggl(\begin{matrix} 0\\
K_{m}\end{matrix}\biggr)$ is Ding projective by Lemma \ref{lem:
3.6}. So $K_{m}$ is Ding projective by Theorem \ref{thm: 3.4}. Hence
$Dpd(N)\leq m$. Thus $lDPD(B)\leq m$.

Let $Y$ be any left $A$-module. By Theorem \ref{thm: 3.8}, $Dpd(Y)\leq Dpd\biggl(\begin{matrix} Y\\
0\end{matrix}\biggr)\leq lDPD(T)= m$. Thus $lDPD(A)\leq m$. It
follows that max$\{lDPD(A),lDPD(B), 1\}\leq lDPD(T)$.

Next we prove that $lDPD(T)\leq$ max$\{lDPD(A)+1,lDPD(B)\}$.

We may assume that max$\{lDPD(A)+1,lDPD(B)\}<\infty$. Then
$lDPD(B)<\infty$. By Theorem \ref{thm: 3.8}, for any left $T$-module $M=\biggl(\begin{matrix} M_{1}\\
M_{2}\end{matrix}\biggr)_{\varphi^{M}}$, we have
$$Dpd(M)\leq max\{Dpd(M_{1})+1,Dpd(M_{2})\}\leq
max\{lDPD(A)+1,lDPD(B)\}.$$ So $lDPD(T)\leq$
max$\{lDPD(A)+1,lDPD(B)\}$. This completes the proof.
\end{proof}
\begin{cor} \label{cor: 3.10}Let $R$ be a ring and $T(R)=\biggl(\begin{matrix} R&0\\
R&R \end{matrix}\biggr)$.\begin{enumerate}\item If $lDPD(R)<\infty$ and $M=\biggl(\begin{matrix} M_{1}\\
M_{2}\end{matrix}\biggr)_{\varphi^{M}}$ is a  left $T(R)$-module,
then $$max\{Dpd(M_{1}),Dpd(M_{2})\}\leq Dpd(M)\leq
max\{Dpd(M_{1})+1,Dpd(M_{2})\}.$$\item $max\{lDPD(R), 1\}\leq
lDPD(T(R))\leq lDPD(R)+1.$
\end{enumerate}
\end{cor}
\begin{proof}It follows from Theorems \ref{thm: 3.8} and \ref{thm: 3.9}.
\end{proof}
\begin{rem} \label{rem: 3.11}{\rm Given a left $R$-module $X$, let $Gpd(X)$ denote $\inf\{n$: there
exists an exact sequence $0\rightarrow G_{n}\rightarrow\cdots
\rightarrow G_{1}\rightarrow G_{0}\rightarrow X\rightarrow 0$ of
left $R$-modules with each $G_{i}$ Gorenstein projective\} and call
$Gpd(X)$ the  \emph{Gorenstein projective dimension} of $X$
\cite{H}. If no such $n$ exists, set $Gpd(X)$ = $\infty$. Put
$lGPD(R)$ = sup\{$Gpd(X): X$ is any left $R$-module\} and call
$lGPD(R)$  the \emph{left global Gorenstein projective dimension} of
$R$. Similar to the proofs of Theorems \ref{thm: 3.4}, \ref{thm:
3.8} and \ref{thm: 3.9}, one can obtain that
\begin{enumerate}\item If  $U_{A}$ has finite flat dimension and $_{B}U$ has
finite projective dimension, then a  left $T$-module $M=\biggl(\begin{matrix} M_{1}\\
M_{2}\end{matrix}\biggr)_{\varphi^{M}}$ is Gorenstein projective if
and only if $M_{1}$ is a Gorenstein projective left $A$-module,
$M_{2}/\im(\varphi^{M})$ is a Gorenstein projective left $B$-module
and $\varphi^{M}$ is a monomorphism.
\item If $lGPD(B)<\infty$, $U_{A}$ has finite flat dimension, $_{B}U$ is projective, $M=\biggl(\begin{matrix} M_{1}\\
M_{2}\end{matrix}\biggr)_{\varphi^{M}}$ is a left $T$-module, then
$$max\{Gpd(M_{1}),Gpd(M_{2})\}\leq Gpd(M)\leq max\{Gpd(M_{1})+1,Gpd(M_{2})\}.$$
\item If $U_{A}$ has finite flat dimension and $_{B}U\neq 0$ is projective, then
$$max\{lGPD(A),lGPD(B), 1\}\leq lGPD(T)\leq max\{lGPD(A)+1,lGPD(B)\}.$$
\end{enumerate}}
\end{rem}
\bigskip
\section {\bf Ding injective modules and dimensions}
\bigskip
\begin{lem} \label{lem: 4.1}The following conditions are equivalent:\begin{enumerate}\item
$X$ is a Ding injective right $R$-module. \item There is an exact
sequence $\cdots\rightarrow E^{-1}\rightarrow E^{0}\rightarrow
E^{1}\rightarrow E^{2}\rightarrow \cdots$ of injective right
$R$-modules such that $X=\ker(E^{0}\rightarrow E^{1})$, which is
$\Hom_{R}(G,-)$-exact for each  right $R$-module $G$ with finite
$FP$-injective dimension.
\end{enumerate}
\end{lem}
\begin{proof}The proof is dual to that of Lemma \ref{lem: 3.1}.
\end{proof}
Recall that $R$ is  a \emph{right coherent ring} \cite{L} if every
finitely generated  right ideal is finitely presented.
\begin{lem} \label{lem: 4.2}Let $A$ and $B$ be right coherent rings, $U_{A}$ be finitely
presented and $_{B}U$ have finite flat dimension. If $G$ is an
$FP$-injective right $A$-module, then the right $B$-module
$\Hom_{A}(U,G)$ has finite FP-injective dimension.
\end{lem}
\begin{proof}Since $A$ is right coherent and $G$ is an
$FP$-injective right $A$-module, $G^{+}$ is a flat left $A$-module
by \cite[Theorem 2.2]{F}. So $fd(U\otimes_{A} G^{+})<\infty$ by
Lemma \ref{lem: 3.2} since $fd(_{B}U)<\infty$. By \cite[Lemma
3.60]{R}, $\Hom_{A}(U,G)^{+}\cong U\otimes_{A} G^{+}$ since $U_{A}$
is finitely presented. Hence $fd(\Hom_{A}(U,G)^{+})=fd(U\otimes_{A}
G^{+})<\infty$. Since $B$ is right coherent,
$FP-id(\Hom_{A}(U,G))=fd(\Hom_{A}(U,G)^{+})<\infty$ by \cite[Theorem
2.2]{F}.
\end{proof}
Let $W=(W_{1}, W_{2})_{\varphi_{W}}$ be a  right $T$-module. Then
$W^{+}=\biggl(\begin{matrix}
W_{1}^{+}\\W_{2}^{+}\end{matrix}\biggr)_{\varphi^{W^{+}}}$ is a
character left $T$-module where $\varphi^{W^{+}}: U\otimes_{A}
W^{+}_{1}\rightarrow W^{+}_{2}$ is defined by
$\varphi^{W^{+}}(u\otimes f)(x)=f(\varphi_{W}(x\otimes u))$ for any
$f\in W_{1}^{+}$, $u \in U$ and $x\in W_{2}$.

\begin{lem} \label{lem: 4.3}Let $W=(W_{1}, W_{2})_{\varphi_{W}}$ be a  right $T$-module.
\begin{enumerate}\item $W$ is an injective right $T$-module if and only if
$W_{1}$  is an injective right $A$-module,
$\ker(\widetilde{{\varphi_{W}}})$  is an injective right $B$-module
and $\widetilde{{\varphi_{W}}}$ is an epimorphism.
\item If $T$ is a right coherent ring and $U_{A}$ is finitely
presented, then $W$ is an $FP$-injective right $T$-module if and
only if  $W_{1}$  is an $FP$-injective right $A$-module,
$\ker(\widetilde{{\varphi_{W}}})$  is an $FP$-injective right
$B$-module and $\widetilde{{\varphi_{W}}}$ is an
epimorphism.\end{enumerate}
\end{lem}
\begin{proof}(1) follows from \cite[Proposition 5.1]{HV2} and \cite[p.956]{AS}.

(2) By \cite[Theorem 4.2]{HMV}, both $A$ and $B$ are right coherent
rings.

The exact sequence
$0\rightarrow\ker(\widetilde{{\varphi_{W}}})\rightarrow
W_{2}\stackrel{\widetilde{{\varphi_{W}}}}\rightarrow \Hom_{A}(U,
W_{1})$ induces the exact sequence
$\Hom_{A}(U,W_{1})^{+}\stackrel{\widetilde{{\varphi_{W}}}^{+}}\rightarrow
W^{+}_{2}\rightarrow(\ker(\widetilde{{\varphi_{W}}}))^{+}\rightarrow
0$. By \cite[Lemma 3.60]{R} or \cite[Lemma 1.2.11]{GT},
$\Hom_{A}(U,W_{1})^{+}\cong U\otimes_{A} W_{1}^{+}$. Then we get the
following commutative diagram with exact rows:
$$\xymatrix{\Hom_{A}(U,W_{1})^{+}\ar[d]_{\cong}\ar[rr]^{\widetilde{{\varphi_{W}}}^{+}}&&W^{+}_{2}\ar@{=}[d]\ar[r]&(\ker(\widetilde{{\varphi_{W}}}))^{+}\ar[d]_{\cong}\ar[r]&0\\
U\otimes_{A}W^{+}_{1}
\ar[rr]^{\varphi^{W^{+}}}&&W^{+}_{2}\ar[r]&W^{+}_{2}/\im(\varphi^{W^{+}})\ar[r]&0.}$$

Since $T$ is a right coherent ring, $W=(W_{1}, W_{2})_{\varphi_{W}}$
is an $FP$-injective right $T$-module if and only if
$W^{+}=\biggl(\begin{matrix}
W_{1}^{+}\\W_{2}^{+}\end{matrix}\biggr)_{\varphi^{W^{+}}}$ is a flat
left $T$-module by \cite[Theorem 2.2]{F} if and only if $W^{+}_{1}$
and  $W^{+}_{2}/\im(\varphi^{W^{+}})$ are flat and $\varphi^{W^{+}}:
U\otimes W^{+}_{1}\rightarrow W^{+}_{2}$ is a monomorphism by Lemma
\ref{lem: 3.3} if and only if $W^{+}_{1}$ and
$(\ker(\widetilde{{\varphi_{W}}}))^{+}$ are flat and
$(\widetilde{{\varphi_{W}}})^{+}: \Hom_{A}(U,W_{1})^{+}\rightarrow
W^{+}_{2}$ is a monomorphism by the diagram above if and only if
$W_{1}$ and $\ker(\widetilde{{\varphi_{W}}})$ are $FP$-injective and
$\widetilde{{\varphi_{W}}}: W_{2}\rightarrow \Hom_{A}(U, W_{1})$ is
an epimorphism by \cite[Theorem 2.2]{F}.
\end{proof}
The following theorem characterizes explicitly the structure of a
Ding injective right $T$-module.
\begin{thm} \label{thm: 4.4}Let $T$ be a right coherent ring, $_{B}U$ have finite flat dimension, $U_{A}$ be finitely presented  and have finite projective or $FP$-injective dimension.
 The following conditions  are equivalent for a right $T$-module
 $W=(W_{1}, W_{2})_{\varphi_{W}}$:\begin{enumerate}\item $W$ is a Ding injective right
$T$-module.\item $W_{1}$ is a Ding injective right $A$-module,
$\ker(\widetilde{{\varphi_{W}}})$ is a Ding injective right
$B$-module and $\widetilde{{\varphi_{W}}}$ is an epimorphism.
\end{enumerate}In this case,  $\Hom_{A}(U,W_{1})$ is Ding injective if and only if
$W_{2}$ is Ding injective.
\end{thm}
\begin{proof}The proof is dual to that of Theorem \ref{thm: 3.4} by
using Lemmas \ref{lem: 4.1}, \ref{lem: 4.2}, \ref{lem: 4.3} and
\cite[Theorem 2.8]{YLL}.
\end{proof}
\begin{cor} \label{cor: 4.5}Let $R$ be a right coherent ring, $T(R)=\biggl(\begin{matrix} R&0\\
R&R \end{matrix}\biggr)$ and $W=(W_{1}, W_{2})_{\varphi_{W}}$ be a
right $T(R)$-module. The following conditions are
equivalent:\begin{enumerate}\item $W$ is a Ding injective right
$T(R)$-module.\item $W_{1}$ and $\ker(\widetilde{{\varphi_{W}}})$
are Ding injective right $R$-modules, and
$\widetilde{{\varphi_{W}}}$ is an epimorphism.\item $W_{2}$ and
$\ker(\widetilde{{\varphi_{W}}})$ are Ding injective right
$R$-modules, and $\widetilde{{\varphi_{W}}}$ is an epimorphism.
\end{enumerate}
\end{cor}
\begin{proof} By \cite[Corollary 4.5]{HMV} or \cite[Theorem 2.3]{LC}, $T(R)$ is a right coherent ring.
So the result is an immediate consequence of Theorem \ref{thm: 4.4}.
\end{proof}
Next we consider Ding injective dimensions of modules over formal
triangular matrix rings.

Given a right $R$-module $X$, let $Did(X)$ denote $\inf\{n$: there
exists an exact sequence $0\rightarrow X \rightarrow
H^{0}\rightarrow H^{1}\cdots\rightarrow H^{n}\rightarrow 0$ of right
$R$-modules with each $H^{i}$ Ding injective\} and call $Did(X)$ the
\emph{Ding injective dimension} of $X$ \cite{Y}. If no such $n$
exists, set $Did(X)$ = $\infty$. Put $rDID(R)$ = sup\{$Did(W): W$ is
any right $R$-module\}, and call $rDID(R)$ the \emph{right global
Ding injective dimension} of $R$.
\begin{lem} \label{lem: 4.6}The following conditions are
equivalent for a right $R$-module $X$:
\begin{enumerate}\item Did$(X)\leq n$. \item For any  exact sequence
$0\rightarrow X \rightarrow E^{0}\rightarrow
E^{1}\rightarrow\cdots\rightarrow L^{n}\rightarrow 0$ with  each
$E^{i}$  Ding injective, $L^{n}$ is Ding injective.
\end{enumerate}
\end{lem}
\begin{proof}The proof is dual to that of Lemma \ref{lem: 3.6}.
\end{proof}
\begin{lem} \label{lem: 4.7}Let $rDID(B)<\infty$, $_{B}U$ be flat, $U_{A}$ have finite projective or $FP$-injective dimension.
If $H$ is a Ding injective right $A$-module, then $\Hom_{A}(U,H)$ is
a Ding injective right $B$-module.
\end{lem}
\begin{proof}There is an exact sequence of injective right $A$-modules $$\Lambda: \cdots\rightarrow
E^{-1}\rightarrow E^{0}\rightarrow E^{1}\rightarrow E^{2}\rightarrow
\cdots$$ with $H=\ker(E^{0}\rightarrow E^{1})$, which remains exact
after applying $\Hom_{A}(H,-)$ for each $FP$-injective right
$A$-module $H$. Since $_{B}U$ is flat, each $\Hom_{A}(U, E^{i})$ is
injective. By \cite[Lemma 2.5]{EIT} and Lemma \ref{lem: 4.1}, we
obtain the exact sequence
$$\Hom_{A}(U, \Lambda): \cdots\rightarrow \Hom_{A}(U, E^{-1})\rightarrow \Hom_{A}(U,
E^{0})\rightarrow \Hom_{A}(U, E^{1})\rightarrow \cdots$$ of
injective right $B$-modules with $\Hom_{A}(U,H)\cong\ker( \Hom(U,
E^{0})\rightarrow \Hom(U, E^{1}))$.

For each $FP$-injective right $B$-module $L$, we claim that
$pd(L)<\infty$. In fact, let $rDID(B)=n<\infty$, then for any right
$B$-module $N$, there is an exact sequence $0\rightarrow
N\rightarrow H^{0}\rightarrow H^{1}\rightarrow\cdots \rightarrow
H^{n}\rightarrow 0$ with each $H^{i}$ Ding injective. It follows
that $\Ext_{B}^{n+1}(L,N)\cong \Ext_{B}^{1}(L,H^{n})=0$ by
\cite[Lemma 2.3(1)]{MD}. Thus $pd(L)\leq n$. Hence $\Hom_{B}(L,
\Hom_{A}(U,\Lambda))$ is exact by \cite[Lemma 2.5]{EIT}. So
$\Hom_{A}(U,H)$ is a Ding injective right $B$-module.
\end{proof}
\begin{thm} \label{thm: 4.8}Let $T$ be a right coherent ring, $rDID(B)<\infty$,
$_{B}U$ be flat, $U_{A}$ be finitely presented and have finite
projective or $FP$-injective dimension. If $W=(W_{1},
W_{2})_{\varphi_{W}}$ is a right $T$-module, then
$$max\{Did(W_{1}),Did(W_{2})\}\leq Did(W)\leq
max\{Did(W_{1})+1,Did(W_{2})\}.$$
\end{thm}
\begin{proof}The proof is dual to that of Theorem \ref{thm: 3.8} by using  Theorem \ref{thm: 4.4},
Lemmas \ref{lem: 4.6} and \ref{lem: 4.7}.
\end{proof}
The following theorem gives an estimate of the right global Ding
injective dimension of a formal triangular matrix ring.
\begin{thm} \label{thm: 4.9}Let $T$ be a right coherent ring, $_{B}U\neq 0$ be
flat, $U_{A}$ be finitely presented and have finite projective or
$FP$-injective dimension. Then
$$max\{rDID(A),rDID(B), 1\}\leq rDID(T)\leq
max\{rDID(A)+1,rDID(B)\}.$$
\end{thm}
\begin{proof}The proof is dual to that of Theorem \ref{thm: 3.9} by using Theorems \ref{thm: 4.4}, \ref{thm: 4.8} and Lemma \ref{lem: 4.6}.
\end{proof}
\begin{cor} \label{cor: 4.10}Let $R$ be a right coherent ring and $T(R)=\biggl(\begin{matrix} R&0\\
R&R \end{matrix}\biggr)$.\begin{enumerate}\item If $rDID(R)<\infty$
and $W=(W_{1}, W_{2})_{\varphi_{W}}$ is a  right $T(R)$-module, then
$$max\{Did(W_{1}),Did(W_{2})\}\leq Did(W)\leq
max\{Did(W_{1})+1,Did(W_{2})\}.$$\item $max\{rDID(R), 1\}\leq
rDID(T(R))\leq rDID(R)+1.$
\end{enumerate}
\end{cor}
\begin{proof}It follows from Theorems \ref{thm: 4.8} and \ref{thm: 4.9}.
\end{proof}
\begin{rem} \label{rem: 4.11}{\rm Given a right $R$-module $X$, let $Gid(X)$ denote $\inf\{n$: there
is an exact sequence $0\rightarrow X \rightarrow H^{0}\rightarrow
H^{1}\cdots\rightarrow H^{n}\rightarrow 0$ of right $R$-modules with
each $H^{i}$ Gorenstein injective\} and call $Gid(X)$ the
\emph{Gorenstein injective dimension} of $X$ \cite{H}. If no such
$n$ exists, set $Gid(X)$ = $\infty$. Put $rGID(R)$ = sup\{$Gid(X):
X$ is any right $R$-module\}, and call $rGID(R)$ the \emph{right
global Gorenstein injective dimension} of $R$. Similar to Theorems
\ref{thm: 4.4}, \ref{thm: 4.8} and \ref{thm: 4.9}, we have
\begin{enumerate}\item If $_{B}U$ has finite flat
dimension, $U_{A}$ has finite projective or injective dimension,
then a right $T$-module $W=(W_{1}, W_{2})_{\varphi_{W}}$ is
Gorenstein injective if and only if $W_{1}$ is a Gorenstein
injective right $A$-module, $\ker(\widetilde{{\varphi_{W}}})$ is a
Gorenstein injective right $B$-module and
$\widetilde{{\varphi_{W}}}$ is an epimorphism.
\item If $rGID(B)<\infty$, $_{B}U$ is flat, $U_{A}$ has finite projective or injective dimension, $W=(W_{1}, W_{2})_{\varphi_{W}}$ is a right $T$-module, then
$$max\{Gid(W_{1}),Gid(W_{2})\}\leq Gid(W)\leq max\{Gid(W_{1})+1,Gid(W_{2})\}.$$
\item If $_{B}U\neq 0$ is flat, $U_{A}$ has finite projective or injective dimension, then
$$max\{rGID(A),rGID(B), 1\}\leq rGID(T)\leq max\{rGID(A)+1,rGID(B)\}.$$
\end{enumerate}}
\end{rem}
\bigskip
\centerline {\bf ACKNOWLEDGEMENTS}
\bigskip
This research was supported by NSFC (11771202) and Nanjing Institute
of Technology of China (CKJA201707, JCYJ201842).

\end{document}